\documentclass[11pt]{amsart}
\usepackage{amsfonts,amssymb,amsmath,amsthm}
\usepackage{url}
\usepackage{enumerate}
\usepackage[dvipdfmx]{graphicx}

\newtheorem{theorem}{Theorem}[section]

\newtheorem{lemma}[theorem]{Lemma}

\newtheorem{proposition}[theorem]{Proposition}
\newtheorem{corollary}[theorem]{Corollary}
\theoremstyle{definition}
\newtheorem{definition}[theorem]{Definition}
\newtheorem{remark}[theorem]{Remark}
\numberwithin{equation}{section}

\newtheorem{case}{Case}

\author{Shohei Satake}
\address{
Graduate School of System Informatics, Kobe University \\
Rokkodai 1-1, Nada, Kobe, 657-8501, JAPAN}
\email{155x601x@stu.kobe-u.ac.jp}
\thanks{The author is supported by Grant-in-Aid for JSPS Fellows 18J11282 of the Japan Society for the Promotion of Science.}



\begin{document}

\title[On expander Cayley graphs from Galois rings ]
{On expander Cayley graphs from Galois rings}

\maketitle

\begin{abstract}
In this paper, we study new Cayley graphs over the additive group of Galois rings.
First we prove that they are expander graphs by using a Weil-Carlitz-Uchiyama type estimation of character sums for Galois rings.
We also show that Cayley graphs from Galois rings of characteristic $4$ form a new infinite family of Ramanujan graphs by an elementary eigenvalue estimation. 
Moreover some other spectral properties of our graphs are also discussed.
\end{abstract}

\section{Introduction}
\label{intro}
For a group $\Gamma$ and a subset $S\subset \Gamma \setminus\{0\}$ such that $S=-S$, the {\it Cayley graph} $Cay(\Gamma, S)$ over $\Gamma$ defined by $S$ 
is the graph on the vertex set $\Gamma$ such that two vertices $x$ and $y$ are adjacent if $x-y \in S$.
Here we denote the operation of $\Gamma$ additively and $-S$ is the set of inverses of elements in $S$.
In this paper we deal with Cayley graphs over the additive group of certain finite commutative rings called {\it Galois rings}; see Section~\ref{sec:Galois}.
\par First, we show that they are {\it expander graphs} whose degrees are small powers of the number of vertices.
Especially Cayley graphs constructed from Galois rings of characteristic $4$ form a new infinite family of {\it Ramanujan graphs}.
To explain expander graphs and Ramanujan graphs, we need to explain the adjacency matrix of graphs.
The {\it adjacency matrix} of a graph $G$ is the $0$-$1$ matrix whose rows and columns are indexed by its vertices such that the $(u, v)$-entry is $1$ if and only if $u$ and $v$ are adjacent.
For simplicity, let {\it eigenvalues of $G$} denote eigenvalues of the adjacency matrix of $G$.
If $G$ is a $d$-regular graph, by the Perron-Frobenius theorem (e.g. \cite[Chapter 31]{LW01}), the largest eigenvalue is $d$ and all eigenvalues of $G$ are in the interval $[-d, d]$. 
And $-d$ is an eigenvalue of $G$ if and only if $G$ is bipartite.
Now we are ready to explain expander graphs. 
A $d$-regular graph $G$ is called an {\it expander graph} if its second largest eigenvalue is smaller than the largest eigenvalue $d$ (see e.g. \cite{HLW06}).
Let $\lambda(G)$ be the largest absolute value of eigenvalues of $G$ other than $\pm d$. 
Then a $d$-regular graph $G$ is called a {\it Ramanujan graph} if $\lambda(G) \leq 2\sqrt{d-1}$ (see \cite{LPS88}).
By the Alon-Boppana bound (see e.g. \cite{LPS88}), Ramanujan graphs can be regarded as the best possible expander graphs. 
Expander graphs have wide applications (for details, see e.g. \cite{HLW06}) and explicit constructions of expander graphs 
are very important and interesting from the view of mathematics and applications.
In these three decades, many constructions are found in various areas of mathematics.
Especially, as pointed out by Chee-Ling~\cite{CL02}, expander graphs with symmetry are important in applications. 
So constructions based on Cayley graphs are preferable since Cayley graphs have vertex-transitivity.
For constructions based on Cayley graphs, see e.g. \cite{LZ18}.
\par We note that our expander Cayley graphs have the similar sparseness to those of the Chung's graphs~\cite{C89} and some generalized graphs in \cite{HL18}, \cite{LWWZ14}, \cite{RM17} and \cite{S93}. 
And it is worth noting that our construction of Ramanujan graphs from Galois rings of characteristic $4$ uses only elementary discussions and fundamental facts of Galois rings. 
We remark that known constructions of Ramanujan graphs are usually based on deep results or methods in number theory and related algebraic areas.
\par The rest of this paper is organized as follows. 
In Section~\ref{sec:Galois}, we give the definition of Galois rings and some known basic facts. 
Section~\ref{sec:Cayley}, \ref{sec:Const} and \ref{sec:spec} are the main part of this paper.
In Section~\ref{sec:Cayley}, we construct Cayley graphs from Galois rings and show that they are expander graphs by a Weil-Carlitz-Uchiyama type estimation for Galois rings (\cite{KHC95}).
We also discuss the girth for some cases.
In Section~\ref{sec:Const}, we focus on Cayley graphs from Galois rings of characteristic $4$ and show that they are Ramanujan graphs by an elementary character sum estimation based on ~\cite{BHK92}.
In Section~\ref{sec:spec}, we also prove some other interesting spectral property of our Ramanujan graphs in Section~\ref{sec:Const}.
At last, in Section~\ref{sec:concl}, we give a concluding remark.

\section{Galois rings}
\label{sec:Galois}
For details of Galois rings, see the textbook \cite[Chapter 14]{W12}. 
We will also refer some facts denoted in \cite[Section III, IV]{BHK92} and \cite[Section III]{HKCSS94}. 
\par Let $p$ be a prime and $e$ be a positive integer. 
Let $\mathbb{Z}_{p^e}$ be the residue ring $\mathbb{Z}/p^e\mathbb{Z}$ and $\mathbb{Z}_{p^e}[x]$ the polynomial ring over $\mathbb{Z}_{p^e}$.
Define $\rho: \mathbb{Z}_{p^e}\rightarrow \mathbb{Z}_{p}$ as the reduction map.
For a polynomial $f(x)=a_0+a_1x+\cdots+a_kx^k \in \mathbb{Z}_{p^e}[x]$, let $\rho(f(x)):=\rho(a_0)+\rho(a_1)x+\cdots+\rho(a_k)x^k \in \mathbb{Z}_{p}[x]$.
Then $f(x) \in \mathbb{Z}_{p^e}[x]$ is called {\it monic basic irreducible} if $\rho(f(x)) \in \mathbb{Z}_{p}[x]$ is a monic irreducible polynomial.
For each $r\geq 1$, the Galois ring of characteristic $p^e$ and order $p^{er}$ is defined as follows.

\begin{definition}
Let $f(x) \in \mathbb{Z}_{p^e}[x]$ be a monic basic irreducible polynomial of degree $r$.
Then the {\it Galois ring} $GR(p^{e}, p^{er})$ of {\it characteristic} $p^e$ and {\it order} $p^{er}$ is defined as follows:
\[
GR(p^{e}, p^{er}):=\mathbb{Z}_{p^e}[x]/(f(x))
\]
Here $(f(x))$ is the principle ideal of $\mathbb{Z}_{p^e}[x]$ generated by $f(x)$ and order and characteristic are defined in the same manner as the case of finite fields.
\end{definition} 
Let $GR^+(p^{e}, p^{er})$ and $GR^*(p^{e}, p^{er})$ be the additive and multiplicative group of $GR(p^{e}, p^{er})$, respectively.
The group structures of $GR^+(p^{e}, p^{er})$ and $GR^*(p^{e}, p^{er})$ were completely determined \cite[Section 14.1, 14.4]{W12}.
$GR(p^{e}, p^{er})$ is a free $\mathbb{Z}_{p^e}$-module of rank $r$.
And $GR^*(p^{e}, p^{er})$ contains the cyclic group, say $G_1$, of order $p^r-1$.
Let $\xi$ be a generator of $G_1$. Then the following theorem is known.
\begin{theorem}[Theorem 14.8 in \cite{W12}]
For the Galois ring $GR(p^e, p^{er})$, the following statements hold.
\begin{enumerate}
\item[$(1)$] Each element of $GR(p^{e}, p^{er})$ can be uniquely expressed by the form $a_0+a_1\xi+\cdots+a_{r-1}\xi^{r-1}$ 
where $a_0, a_1,\ldots, a_{r-1} \in \mathbb{Z}_{p^e}$.

\item[$(2)$] Each element of $GR(p^{e}, p^{er})$ can be uniquely expressed by the form $b_0+b_1p+\cdots+b_{e-1}p^{e-1}$ where $b_0, b_1,\ldots, b_{e-1} \in G_1 \cup \{0\}$.
Moreover, $b_0+b_1p+\cdots+b_{e-1}p^{e-1}$ is invertible if and only if $b_0 \neq 0$. 
\end{enumerate} 
\end{theorem} 
We can also define the trace function from $GR(p^{e}, p^{er})$ to $\mathbb{Z}_{p^e}$. 
Let $\phi: GR(p^{e}, p^{er}) \rightarrow GR(p^{e}, p^{er})$ be the map such that $a_0+a_1\xi+\cdots+a_{r-1}\xi^{r-1}$ is mapped to $a_0+a_1\xi^{p}+\cdots+a_{r-1}\xi^{p(r-1)}$.
The map $\phi$ is actually a ring automorphism of $GR(p^{e}, p^{er})$ and it is called the {\it generalized Frobenius automorphism}.
Then the {\it trace function} $T$ from $GR^+(p^{e}, p^{er})$ to $\mathbb{Z}_{p^e}$ is defined by $T(x)=x+\phi(x)+\cdots+\phi^{r-1}(x)$ for each $x \in GR(p^{e}, p^{er})$.
Actually $T$ is a surjective linear map. \\
　Let $\mu: GR(p^e, p^{er}) \rightarrow GR(p^e, p^{er})/pGR(p^e, p^{er})$ be the natural projection map which is a ring homomorphism. 
Since $pGR(p^e, p^{er})$ is the unique maximal ideal of $GR(p^e, p^{er})$, the residue ring $GR(p^e, p^{er})/pGR(p^e, p^{er})$ is a finite field isomorphic to $\mathbb{F}_{p^r}$.
And $\theta:=\mu(\xi)$ is a primitive element of $\mathbb{F}_{p^r}$.

\section{Cayley graphs from Galois rings}
\label{sec:Cayley}
In this section, we define the following Cayley graph over $GR^+(p^e, p^{er})$.
\begin{definition}
Let $e, r\geq 2$. 
For the prime $p=2$, we define the Cayley graph $H_{2^{e}, 2^{er}}:=Cay(GR^+(2^e, 2^{er}), G_1\cup -G_1)$.
And for odd primes $p \geq 3$, we also define the Cayley graph $H_{p^e, p^{er}}:=Cay(GR^+(p^e, p^{er}), G_1)$.
\end{definition}
\label{def:Cayley}
Note that for odd primes $p$, $H_{p^e, p^{er}}$ is well-defined since $G_1=-G_1$.
In fact, since $p$ is odd and $\xi^{p^r-1}=1$, we see that $(\xi^{(p^r-1)/2}-1)(\xi^{(p^r-1)/2}+1)=1$ 
and thus $-1=\xi^{(p^r-1)/2}$ because $\xi^{(p^r-1)/2}-1$ is invertible by the following lemma.

\begin{lemma}
\label{lem:unit}
Let $p$ be a prime and $e \geq 2$ an integer. Then for each $0\leq i \neq j \leq p^{r}-2$, $\xi^i-\xi^j$ is invertible.
\end{lemma}
\begin{proof}
It suffices to prove that $1-\xi^i$ is invertible for every $1\leq i \leq p^{r}-2$.
Suppose that $1-\xi^i \in pGR(p^e, p^{er})$ for some $1\leq i \leq p^{r}-2$. 
Then $\theta^i=1$, which contradicts the fact that $\theta$ is a primitive element of $\mathbb{F}_{p^r}$.
\end{proof}
From Definition~\ref{def:Cayley}, we can easily get the following proposition.
\begin{proposition}
\label{prop:deg}
\begin{enumerate}
\item[$(1)$] $H_{2^e, 2^{er}}$ is a $(2^{r+1}-2)$-regular graph with $2^{er}$ vertices. 
\item[$(2)$] $H_{p^e, p^{er}}$ is a $(p^{r}-1)$-regular graph with $p^{er}$ vertices.
\end{enumerate}
\end{proposition} 
\begin{proof}
Since $(2)$ is obvious from the definition, we prove $(1)$. 
Now we know the size of $H_{2^e, 2^{er}}$. The degree is equal to the size of $G_1\cup -G_1$.
Since $|G_1|=2^r-1$, we shall prove that $G_1$ and $-G_1$ are disjoint.
If $\xi^k=-\xi^{l}$ for some $0 \leq k, l \leq 2^r-2$, then $\xi^{k-l}=-1$.
When $k=l$, this contradicts the fact that the characteristic of $GR(2^{e}, 2^{er})$ is $2^e>2$.
When $k \neq l$, we see that $-1 \in G_1$ and it generates the cyclic subgroup of order $2$, which contradicts the fact that the order of $G_1$ is odd. \\
\end{proof}
Now we get the following theorem.
\begin{theorem}
\label{thm:H1}
\begin{enumerate}
\item[$(1)$]
All eigenvalues of the graph $H_{2^e, 2^{er}}$ except for the degree $2^{r+1}-2$ are in the interval
$[-2^{e+\frac{r}{2}}+2^{\frac{r}{2}+1}-2,\; 2^{e+\frac{r}{2}}-2^{\frac{r}{2}+1}+2]$.

\item[$(2)$]
For every odd prime $p\geq 3$, all eigenvalues of the graph $H_{p^e, p^{er}}$ except for the degree $p^{r}-1$ are in the interval
$[-p^{e+\frac{r}{2}-1}+p^{\frac{r}{2}}-1,\; p^{e+\frac{r}{2}-1}-p^{\frac{r}{2}}+1]$.
\end{enumerate}
\end{theorem}
This theorem can be proved by the following result which is an analogue of a Weil-Carlitz-Uchiyama type estimation of character sums over finite fields.
The following estimation is a special case of the result proved by Kumar-Helleseth-Calderbank~\cite{KHC95}.
\begin{lemma}[Kumar-Helleseth-Calderbank~\cite{KHC95}]
\label{lem:KHC}
For $\alpha=b_0+b_1p+\cdots+b_{e-1}p^{e-1}$ where $b_0, b_1,\ldots, b_{e-1} \in G_1 \cup \{0\}$, 
let $i_\alpha:=\min\{i \mid b_i \neq 0, \; 0\leq i \leq e-1\}$ and $N_{\alpha}:=p^{e-1-i_{\alpha}}$.
Then, for every non-trivial character $\psi$ of $GR^+(p^e, p^{er})$ and $\alpha\neq0$,
\begin{equation}
\biggl|\sum_{s\in G_1}\psi(\alpha s) \biggr| \leq (N_{\alpha}-1)\sqrt{p^r}+1.
\end{equation} 
Especially, $\max_{\alpha \neq 0}N_\alpha=p^{e-1}$.
\end{lemma}

\begin{proof}[Proof of Theorem~\ref{thm:H1}]
It is known that if $\Gamma$ is an abelian group, the following set is the multi-set of $|\Gamma|$ eigenvalues of $Cay(\Gamma, S)$: 
\[
\Bigl\{ \sum_{s\in S}\psi(s) \mid \text{$\psi$ is a character of $\Gamma$} \Bigr\}.
\]
Thus each eigenvalue of $H_{2^e, 2^{er}}$ is expressed by $\sum_{s\in G_1 \cup -G_1}\psi(s)$ for a character $\psi$ of $GR^+(2^e, 2^{er})$.
Similarly, for $p\geq 3$, each eigenvalue of $H_{p^e, p^{er}}$ is expressed by $\sum_{s\in G_1}\psi(s)$ for a character $\psi$ of $GR^+(p^e, p^{er})$.
$(2)$ is a direct consequence of Lemma~\ref{lem:KHC}.
And since $G_1$ and $-G_1$ are disjoint when $p=2$, $(1)$ is also proved by Lemma~\ref{lem:KHC}.
\end{proof}

By Theorem~\ref{thm:H1}, we get the following condition for connectivity.

\begin{corollary}
For each $p$, $H_{2^e, 2^{er}}$ and $H_{p^e, p^{er}}$ are connected if $e<r/2+1$.
\end{corollary}
\begin{proof}
The diameter of a graph $G$ is the maximal length of the shortest paths in $G$.
If $G$ is disconnected, the diameter is defined as $\infty$.
Then it is not so difficult to show that $G$ is connected if its diameter is finite.
For a $d$-regular graph of $G$ on $n$ vertices whose eigenvalues other than $d$ are in the interval $[-\lambda, \lambda]$, Chung~\cite{C89} showed that 
\begin{equation}
(\text{the diameter of $G$}) \leq \frac{\log (n-1)}{\log(\frac{d}{\lambda})}.
\end{equation}
This bound implies that $G$ is connected if $\lambda<d$. By Theorem~\ref{thm:H1}, this is ensured when $e<r/2+1$ for both graphs.
\end{proof}

By Theorem~\ref{thm:H1}, we get many expander graphs which have the sparseness similar to those of the Chung's graphs~\cite{C89} and its generalizations in \cite{HL18}, \cite{LWWZ14}, \cite{RM17} and \cite{S93}.
For example, for each fixed $e$ and $p$, they form the infinite families $\{H_{2^e, 2^{er}}\}_{r\geq 1}$ and $\{H_{p^e, p^{er}}\}_{r\geq 1}$ of expander graphs whose degrees are the $1/e$-th power of the sizes. 
Moreover the following corollary gives sparser expander graphs.
Constructions of graphs which have the similar sparseness are also considered in Alon-Roichman~\cite{AR94}.
\begin{corollary}
\label{prop:H2}
For each rational number $0<\delta\leq 1/2$ and for all $r\geq 1$ such that $\delta r$ is an integer, the following statements hold.
\begin{enumerate}
\item[$(1)$]
The graph $H_{2^{\delta r}, 2^{\delta r^2}}$ is a connected $(2^{r+1}-2)$-regular graph with $2^{\delta r^2}$ vertices such that $\lambda(H_{2^{\delta r}, 2^{\delta r^2}}) \leq 2^{(1/2+\delta)r}-2^{\frac{r}{2}+1}+2$.

\item[$(2)$]
For every odd prime $p$, the graph $H_{p^{\delta r}, p^{\delta r^2}}$ is a connected $(p^{r}-1)$-regular graph with $p^{\delta r^2}$ vertices such that $\lambda(H_{p^{\delta r}, p^{\delta r^2}}) \leq p^{(1/2+\delta)r-1}-p^{\frac{r}{2}}+1$.
\end{enumerate}
\end{corollary}
Moreover, for the graph $H_{2^e, 2^{er}}$, we get the following theorem on {\it girth} of the graphs which is defined as the length of the shortest cycle.

\begin{theorem}
\label{thm:girth}
If $r$ is odd and $e\geq 2$, the graph $H_{2^e, 2^{er}}$ is triangle-free, that is, the girth of $H_{2^e, 2^{er}}$ is more than $4$.
Moreover if $r$ is odd and $e=2$, the girth of $H_{4, 4^{r}}$ is $4$. 
\end{theorem}

\begin{proof}
The second statement follows from the first statement and Theorem 4.2 in \cite{LW01}.
We prove the first statement.
\par In general, $Cay(\Gamma, S)$ is triangle-free if the equation $x+y+z=0$ has no solutions in $S$.
So we shall prove that for every $0\leq j, k, l \leq 2^r-2$, the equation $\xi^j \pm \xi^k \pm \xi ^l=0$ does not hold.
We may assume that $j, k, l$ are distinct since if not, the above inequality does not hold because $\xi^j$, $\xi^k$ and $\xi^l$ are invertible.
Without loss of generality, we may consider the following three cases.
The proof below is very similar to the discussion in \cite[Section III. C]{HKCSS94}.
\begin{case}
\label{case:1}
When $\xi^j-\xi^k\pm \xi^l=0$, then we get $1+\xi^a=\xi^b$ for some  $0\leq a\neq b \leq 2^r-2$.
By squaring the both sides, we obtain $1+2\xi^a+\xi^{2a}=\xi^{2b}$.
On the other hand, we also obtain $1+\xi^{2a}=\xi^{2b}$ since $\phi(1+\xi^a)=\phi(\xi^b)$ and $\phi$ is a ring automorphism.
So $2\xi^a=0$, which contradicts the fact that the characteristic of $GR(2^e, 2^{er})$ is $2^e>2$.
\end{case}

\begin{case}
\label{case:2}
When $\xi^j+\xi^k+\xi^l=0$, then we get $1+\xi^a=-\xi^b$ for some  $0\leq a\neq b \leq 2^r-2$.
Assume that $a\neq 2b$.
Similarly for Case~\ref{case:1}, by squaring the both sides, we obtain $1+2\xi^a+\xi^{2a}=\xi^{2b}$.
On the other hand, we also obtain $1+\xi^{2a}=-\xi^{2b}$ by $\phi$.
So $2(\xi^a-\xi^{2b})=0$, which contradicts by Lemma~\ref{lem:unit}.
\end{case}

\begin{case}
\label{case:3}
Assume that $a=2b$ in Case~\ref{case:2}.
Then $1+\xi^b+\xi^{2b}=0$.
Let $\theta:=\mu(\xi) \in \mathbb{F}_{2^r}$ (recall the definition of $\mu$ in Section~\ref{sec:Galois}).
By applying $\mu$ for the above equality, we get $\theta^{2b}+\theta^b+1=0$.
Thus $\theta^b$ is a zero of the equation $x^2+x+1=0$ in $\mathbb{F}_{2^r}$.
It contradicts the fact that this quadratic equation has no solutions in $\mathbb{F}_{2^r}$ when $r$ is odd.
In fact, $tr(x^2+x+1)=r \neq 0$ for all $x \in \mathbb{F}_{2^r}$ when $r$ is odd.
Here $tr$ is the field trace from $\mathbb{F}_{2^r}$ to $\mathbb{F}_{2}$. 
\end{case}
\end{proof}

\begin{remark}
It seems to be interesting to investigate the girth of $H_{p^e, p^{er}}$ and $H_{2^e, 2^{er}}$ in general cases.
At this point, we did not obtain any result except for the case of $H_{2^e, 2^{er}}$ when $r$ is odd. 
\end{remark}

\section{Ramanujan graphs from Galois rings of characteristic $4$}
\label{sec:Const}
Here we consider the Cayley graph $H_{4, 4^r}=Cay(GR^+(4, 4^r), G_1\cup -G_1)$.
The following theorem shows that $\{H_{4, 4^r}\}_{r\geq 4}$ is an infinite family of Ramanujan graphs.
\begin{theorem}
\label{thm:main}
$H_{4, 4^r}$ is a Ramanujan graph of degree $2^{r+1}-2$ with $4^r$ vertices for all $r\geq 4$.
Moreover all eigenvalues except for the degree are in the interval $[-2^{r/2+1}-2, 2^{r/2+1}+2]$.
\end{theorem}
We can prove it by using Lemma~\ref{lem:KHC}, but here we prove by an elementary eigenvalue estimation based on the discussion in \cite{BHK92}. 
Since $G_1$ and $-G_1$ are disjoint, it suffices to evaluate $\zeta_{\gamma}:=\sum_{s\in G_1}\omega^{T(\gamma s)}$ for each $\gamma \in GR(4, 4^r)$ where $\omega:=\sqrt{-1}$, a primitive $4$th root of unity.
It is easy to see that $\zeta_0=2^r-1$.
Boztas-Hammons-Kumar~\cite[Theorem 4]{BHK92} gave the following evaluation of this sum by using only elementary calculation 
and an observation based on fundamental facts of $GR(4, 4^r)$.

\begin{lemma}[Boztas-Hammons-Kumar~\cite{BHK92}]
\label{lem:zeta}
If $\gamma \in GR^*(4, 4^r)$, then
\begin{equation}
\label{eq:zetaunit}
2^{r/2}-1 \leq |\zeta_{\gamma}|\leq 2^{r/2}+1.
\end{equation}
If $\gamma \in GR(4, 4^r) \setminus GR^*(4, 4^r)$ and $\gamma \neq 0$ then
\begin{equation}
\label{eq:zetanonunit}
\zeta_{\gamma}=-1.
\end{equation}
\end{lemma}
\begin{proof}
First for each $\gamma \in GR^*(4, 4^r)$, the residue classes of $G_1$ in $GR^*(4, 4^r)$ can be expressed as follows:
\begin{equation}
\label{eq:residue}
\gamma G_1,\; -\gamma G_1,\; (1-\xi)\gamma G_1,\; (1-\xi^2)\gamma G_1,\; \ldots,\; (1-\xi^{2^r-2})\gamma G_1.
\end{equation}
and under the above partition, one can also see that
\begin{equation}
\label{eq:nonunit}
GR(4, 4^r)\setminus GR^*(4, 4^r)=2\gamma G_1 \cup \{0\}.
\end{equation}
This is a consequence of Lemma 1 in \cite{BHK92} (see also the equation (4.4) in \cite{BHK92}).
\par Next when $\gamma \in GR^*(4, 4^r)$, 
\begin{align}
\label{eq:zeta1}
\begin{split}
|\zeta_{\gamma}|^2&=\sum_{t=0}^{2^r-2}\sum_{s \in G_1}\omega^{T(\gamma (1-\xi^t)s)}\\
&=\sum_{s \in G_1}\omega^{T(0)}+\sum_{t=1}^{2^r-2}\sum_{s \in G_1}\omega^{T(\gamma (1-\xi^t)s)}\\
&=2^r-1+\sum_{t=1}^{2^r-2}\sum_{s \in G_1}\omega^{T(\gamma (1-\xi^t)s)}.
\end{split}
\end{align}
Then by (\ref{eq:residue}),
\begin{align}
\label{eq:zeta2}
\begin{split}
\sum_{t=1}^{2^r-2}\sum_{s \in G_1}\omega^{T(\gamma (1-\xi^t)s)}&=\sum_{s \in GR^*(4, 4^r)}\omega^{T(s)}-\sum_{s \in G_1}\omega^{T(\gamma s)}-\sum_{s \in -G_1}\omega^{T(\gamma s)}\\
&=\sum_{s \in GR^*(4, 4^r)}\omega^{T(s)}-\zeta_{\gamma}-\overline{\zeta_{\gamma}}.
\end{split}
\end{align}
Here for a complex number $z$, $\overline{z}$ is the complex conjugate of $z$. 
Remark that 
\[
\sum_{s \in -G_1}\omega^{T(\gamma s)}=\sum_{s \in G_1}\omega^{-T(\gamma s)}=\overline{\zeta_{\gamma}}
\]
because $\overline{\omega}=\omega^{-1}=\omega^3$, $\omega^0=\omega^{-0}=1$ and $\omega^2=\omega^{-2}=-1$.
Now, since in the $0$-$2$ sequence $(T(s))_{s\in 2GR(4, 4^r)}$, the frequency of $0$ is equal to the number of occurrences of $2$, we get 
\begin{align}
\label{eq:zeta2-1}
\begin{split}
\sum_{s \in 2GR(4, 4^r)}\omega^{T(s)}=0.
\end{split}
\end{align}
In fact, let $K:=\{s \in 2GR(4, 4^r) \mid T(s)=0\}$. 
It is easy to see that there exists $\alpha \in 2GR(4, 4^r)$ such that $T(\alpha)=2$. 
If $\alpha \in 2GR(4, 4^r)$ and $T(\alpha)=2$, then $T(\alpha+\beta)=2$ for all $\beta \in K$. 
Conversely, if $\alpha'\in 2GR(4, 4^r)$ such that $T(\alpha')=2$, $T(\alpha+(\alpha'-\alpha))=2$ and $\alpha'-\alpha \in K$ since $2GR(4, 4^r)$ is an ideal.
Thus $|\{s \in 2GR(4, 4^r) \mid T(\alpha)=2\}|=|\{\alpha+\beta \mid \beta \in K\}|=|K|$.
Thus by the orthogonal relation of characters and (\ref{eq:zeta2-1}), we obtain
\begin{align}
\label{eq:zeta2-2}
\begin{split}
\sum_{s \in GR^*(4, 4^r)}\omega^{T(s)}&=\sum_{s \in GR(4, 4^r)}\omega^{T(s)}-\sum_{s \in 2GR(4, 4^r)}\omega^{T(s)}\\
&=0-0=0.
\end{split}
\end{align}
Thus by (\ref{eq:zeta1}), (\ref{eq:zeta2}) and (\ref{eq:zeta2-2}),
we get $|1+\zeta_{\gamma}|^2=2^r$, proving the first statement by the triangle inequality.
\par Next when $\gamma \in GR(4, 4^r) \setminus GR^*(4, 4^r)$ such that $\gamma\neq 0$, $\gamma=2 \gamma'$ for some $\gamma' \in GR^*(4, 4^r)$.
As (\ref{eq:zeta1}) and  (\ref{eq:zeta2}), we get
\begin{align}
\label{eq:zeta3}
\begin{split}
|\zeta_{\gamma}|^2&=\sum_{t=0}^{2^r-2}\sum_{s \in G_1}\omega^{T(2\gamma' (1-\xi^t)s)}\\
&=2^r-1+\sum_{t=1}^{2^r-2}\sum_{s \in G_1}\omega^{T(2\gamma' (1-\xi^t)s)},
\end{split}
\end{align}
\begin{align}
\label{eq:zeta4}
\begin{split}
\sum_{t=1}^{2^r-2}\sum_{s \in G_1}\omega^{T(2\gamma' (1-\xi^t)s)}&=\sum_{s \in GR^*(4, 4^r)}\omega^{T(2s)}-\sum_{s \in G_1}\omega^{T(2\gamma' s)}-\sum_{s \in -G_1}\omega^{T(2\gamma' s)}\\
&=\sum_{s \in GR^*(4, 4^r)}(-1)^{T(s)}-\zeta_{\gamma}-\overline{\zeta_{\gamma}}.
\end{split}
\end{align}
And by (\ref{eq:nonunit}), we see that   
\begin{align}
\label{eq:zeta5}
\begin{split} 
\sum_{s \in GR^*(4, 4^r)}(-1)^{T(s)}&=-\sum_{s \in GR(4, 4^r) \setminus GR^*(4, 4^r)}(-1)^{T(s)}\\
&=-\sum_{s \in 2\gamma' G_1 \cup \{0\}}(-1)^{T(s)}\\
&=-\sum_{s \in \gamma' G_1 \cup \{0\}}1^{T(s)}=-2^r.
\end{split}
\end{align}
Thus by (\ref{eq:zeta3}), (\ref{eq:zeta4}) and (\ref{eq:zeta5}), we get $|1+\zeta_{\gamma}|^2=0$, proving the second statement. 
\end{proof}

\begin{remark}
The proofs of (\ref{eq:zeta2-2}) and (\ref{eq:zeta5}) are obtained by a slight modification of the original proof in \cite[p.1105]{BHK92}.
By our proof, we can reduce the facts of Galois rings which we need to use.
\end{remark}

Now we are ready to prove the Theorem~\ref{thm:main}.
\begin{proof}[Proof of Theorem~\ref{thm:main}]
By Lemma~\ref{lem:zeta} we see that 
\begin{equation}
\label{eq:lambda}
\lambda(H_{4, 4^r}) \leq 2^{r/2+1}+2.
\end{equation}
Thus by Proposition~\ref{prop:deg} and (\ref{eq:lambda}),
$H_{4^r}$ is a Ramanujan graph for $r \geq 4$.  
The second statement can also be obtained from Lemma~\ref{lem:zeta}.
\end{proof}

\section{Other spectral properties of $H_{4, 4^r}$}
\label{sec:spec} 
In this section, we show some interesting spectral properties of the graphs $H_{4^r}$. 
At first, we review some definitions.
A graph $G$ is called {\it integral} if every eigenvalue of $G$ is an integer.
Let $E(G):=\sum_{\lambda}|\lambda|$ where $\lambda$ moves over all eigenvalues of $G$.
$E(G)$ is called the {\it energy} of $G$. 
Moreover $G$ is called {\it hyperenergetic} if $G$ has $n$ vertices and $E(G)>2(n-1)$.
Integral graphs and the energy of graphs have been investigated in graph theory and other various areas such as mathematical chemistry and number theory.
For details and related works, see \cite{BCRSS02}, \cite{LSG12} and \cite{LZ18}.   
\par Now we get the following results. 
\begin{theorem}
\label{thm:integr}
For $r\geq 2$, $H_{4, 4^r}$ is integral.    
\end{theorem}
\begin{proof}
Since the real part and imaginary part of $\zeta_{\gamma}$ are integers and all eigenvalue of $H_{4, 4^r}$ must be real, 
the theorem is proved.
\end{proof} 
\begin{theorem}
\label{thm:hypen}
For $r\geq 2$, $H_{4, 4^r}$ is hyperenergetic.   
\end{theorem}
\begin{proof}
Let $n_r$ be the number of vertices of $H_{4, 4^r}$.
By Lemma~\ref{lem:zeta}, there are $|GR^*(4, 4^r)|=4^r-2^r$ eigenvalues such that their absolute values are more than $2^{r/2+1}-1$.
And there are $2^r-1$ eigenvalues which is equal to $-2$ and one eigenvalue $2^r-1$.
Thus, 
\begin{align}
\begin{split}
E(H_{4^r})&\geq |2^{\frac{r}{2}+1}-1|\cdot(4^r-2^r)+|-2|\cdot(2^r-1)+|2^r-1|\cdot 1\\
&=2(n_r)^{\frac{5}{4}}-n_r-2(n_r)^{\frac{3}{4}}+4(n_r)^{\frac{1}{2}}-3
\end{split} 
\end{align}
So $H_{4, 4^r}$ is hyperenergetic if $r\geq 2$.
\end{proof}

\section{Concluding remark}
\label{sec:concl}
At last, we remark that for each invertible element $\gamma \in GR^*(p^e, p^{er})$, we can define two Cayley graphs $H_{2^e, 2^{er}}^\gamma:=Cay(GR^+(2^e, 2^{er}), \gamma G_1 \cup -\gamma G_1)$ 
and $H_{p^e, p^{er}}^\gamma:=Cay(GR^+(p^e, p^{er}), \gamma G_1)$.
For these graphs, we can obtain the same results as Theorem~\ref{thm:H1}, \ref{thm:girth}., \ref{thm:main}.
Especially, we can prove the same statements as Theorem~\ref{thm:integr} and \ref{thm:hypen}.
On the other hand, for each $\gamma \in pGR(p^e, p^{er})$, we can define Cayley digraphs as above.
However, they are not strongly-connected since the multiplicity of the largest eigenvalue is more than $2$ and by the Perron-Frobenius theorem.

\section*{acknowledgements}
We would like to deeply appreciate Masanori Sawa for his valuable comments.
And we greatly appreciate Sanming Zhou for telling us the survey \cite{LZ18}.


\begin{thebibliography}{99}
\bibitem{AR94}
N. Alon, Y. Roichman,
Random Cayley graphs and expanders,
{\it Random Structures Algorithms} {\bf 5} (1994), no. 2, 271--284.

\bibitem{BCRSS02}
K. Bali\'{n}ska, D. Cvetkovi\'{c}, Z. Radosavljevi\'{c}, S. Simi\'{c}, D. Stevanovi\'{c},
A survey on integral graphs,
{\it Univ. Beograd. Publ. Elektrotehn. Fak. Ser. Mat.} {\bf 13} (2002), 42--65. 


\bibitem{BHK92}
S. Boztas, R. Hammons, P. V. Kumar, 
$4$-phase sequences with near-optimum correlation properties,
{\it IEEE Trans. Inform. Theory} {\bf 38} (1992), no. 3 1101--1113.


\bibitem{CL02}
Y. M. Chee, S. Ling,
Highly symmetric expanders,
{\it Finite Fields Appl.} {\bf 8} (2002), no. 3, 294--310. 


\bibitem{C89}
F. R. K. Chung, 
Diameters and eigenvalues,
{\it J. Amer. Math. Soc.} {\bf 2} (1989), no. 2, 187--196. 



\bibitem{HKCSS94}
R. Hammons, P. V. Kumar, A. R. Calderbank, N. J. A. Sloane, P. Sol\'{e},
The $\mathbb{Z}_4$-Linearity of Kerdcck, Preparata, Goethals, and Related Codes,
{\it IEEE Trans. Inform. Theory} {\bf 40} (1994), no. 2, 301-319.

\bibitem{HLW06}
S. Hoory, N. Linial, A. Wigderson, 
Expander graphs and their applications,
{\it Bull. Amer. Math. Soc.} {\bf 43} (2006), no. 4, 439-561.


\bibitem{HL18}
M.-D. Huang, L. Liu, 
Generating sets for the multiplicative groups of algebras over finite fields and expander graphs, 
{\it J. Symbolic Comput.} {\bf 85} (2018), 170--187. 


\bibitem{KHC95}
P. V. Kumar, T. Helleseth, A. R. Calderbank, 
An upper bound for Weil exponential sums over Galois rings and applications,
{\it IEEE Trans. Inform. Theory} {\bf 41} (1995) no. 2, 456-468.

\bibitem{LSG12}
X. Li, Y. Shi, I. Gutman, 
Graph Energy, Springer, New York, 2012. 

\bibitem{LW01} 
J. H. van Lint, R. M. Wilson,
A Course in Combinatorics,
Second edition, Cambridge University Press, 2001.
 

\bibitem{LZ18}
X. Liu, S. Zhou, 
Eigenvalues of Cayley graphs, 
arXiv:1809.09829.


\bibitem{LWWZ14}
M. Lu, D. Wan, L.-P. Wang, X.-D. Zhang,
Algebraic Cayley graphs over finite fields,
{\it Finite Fields Appl.} {\bf 28} (2014), 43--56.
 

\bibitem{LPS88}
A. Lubotzky, R. Phillips, P. Sarnak, 
Ramanujan graphs,
{\it Combinatorica} {\bf 8} (1988), no. 3, 261-277.


\bibitem{RM17}
A. Rasri, Y. Meemark,
Algebraic Cayley graphs over finite local rings,
{\it Finite Fields Appl.} {\bf 48} (2017), 227--240. 


\bibitem{S93}
I. E. Shparlinski, 
On parameters of some graphs from finite fields,
{\it European J. Combin.} {\bf 14} (1993), no. 6, 589--591. 



\bibitem{W12}
Z.-X. Wan,
Finite Fields and Galois Rings, 
World Scientific Publishing Co. Pte. Ltd., 2012.

\end{thebibliography}
\end{document}